\documentclass[sigconf, nonacm]{acmart} 

\AtBeginDocument{%
  }

\usepackage{algorithm}
\usepackage{algpseudocode}
\usepackage{subfig}

\algdef{SE}[SUBALG]{Indent}{EndIndent}{}{\algorithmicend\ }%
\algtext*{Indent}
\algtext*{EndIndent}

\usepackage{listings}
\DeclareMathOperator{\vol}{vol}

\newcommand{\critPts}[0]{{\operatorname{CritPts}}}
\newcommand{\critVals}[0]{{\operatorname{CritVals}}}
\newcommand{\ord}[0]{{\operatorname{ord}}}

\newcommand{\pr}[0]{{\operatorname{pr}}}
\newcommand{\val}[0]{{\operatorname{val}}}

\begin{document}

\title{Exact Volumes of Semi-Algebraic Convex Bodies}

\author{Lakshmi Ramesh}
\orcid{0009-0004-1788-5017}
\authornote{The research of L.R. is funded by the Deutsche Forschungsgemeinschaft (DFG, German Research Foundation) – project number 539663130.}
\affiliation{%
  \institution{Universität Bielefeld}
  \country{Germany}}
\email{lramesh@math.uni-bielefeld.de}

\author{Nicolas Weiss}
\orcid{0009-0000-8252-2045}
\authornote{The research of N.W. is funded by the European Union (ERC, UNIVERSE PLUS, 101118787). Views and opinions expressed are, however, those of the author(s) only and do not necessarily reflect those of the European Union or the European Research Council Executive Agency. Neither the European Union nor the granting authority can be held responsible for them.}
\affiliation{%
  \institution{Max Planck Institute for Mathematics in the Sciences}
  \city{Leipzig}
  \country{Germany}}
\email{nicolas.weiss@mis.mpg.de}

\begin{abstract}

We compute the volumes of convex bodies that are given by inequalities of concave polynomials.
These volumes are found to arbitrary precision thanks to the representation of periods by linear differential equations. Our approach rests on work of Lairez, Mezzarobba, and Safey El Din.
We present a novel method to identify the relevant critical values. Convexity allows us
to reduce the required number of creative telescoping steps by an exponential factor. 
We provide
an implementation based on the \texttt{ore\_algebra} package in SageMath. We present examples computed with our implementation in $2,\; 3$ and $4$ dimensions.

\end{abstract}

\keywords{Semi-algebraic sets; Picard-Fuchs equations; Symbolic-numeric algorithms; Volume computation; Holonomic functions}

\settopmatter{printfolios=true} 
\maketitle

\section{Introduction}\label{sec:intro}

In this paper, we compute volumes of semi-algebraic convex bodies defined by finitely many concave polynomials to arbitrary precision. 
Already in the simplest case of polytopes, it is known that exact volume computation is a \#P-hard problem \cite{volumes-is-difficult}, 
\cite{Khachiyan1993}, \cite{dyer-frieze} 
and is useful in decidability problems, among others \cite{GE2018110}. 
Computing volumes of more general convex semi-algebraic sets with high precision also finds applications in various fields such as portfolio optimization, where the exact volumes of the simplex intersected with parallel hyperplanes and ellipsoids are required \cite{portfolio}, and  geometric statistics, where intersections of a convex body with its translates arise as maximum likelihood estimator (MLE) sets \cite{rameshwahl}. The volume of the MLE set is an important quantity to estimate the barycenter. 

The examples considered in \cite{rameshwahl} are $\ell_p$-balls and their intersections in $\mathbb R^n$. These objects are not only basic semi-algebraic sets, but the polynomials that define them are also concave functions. In this article, we focus on semi-algebraic convex bodies of the form
\begin{equation}\label{eq:class}
    C = \{ x \in \mathbb R^n \mid f_1(x) > 0 , \dots, f_k(x) > 0 \}
\end{equation}
where $f_i \in \mathbb Q[x_1, \dots, x_n]$ are concave. Then, the common positivity locus of the polynomials $f_i$ is the intersection of each of their convex supports, and hence also convex. We will in short refer to this class as \textit{semi-algebraic convex bodies}.


Classically, the volume of semi-algebraic sets is approximated using probabilistic methods, such as the Monte Carlo method. If one samples  uniformly from a box containing the set $C$, the ratio 
\begin{equation*}
    \frac{\# \{\text{samples in $C$}\}}{\# \{\text{all samples}\}} \qquad \xrightarrow{\text{converges to}} \qquad\frac{\vol(\operatorname{C}) }{ \vol(\text{Box})}
\end{equation*}
as the number of samples goes to infinity. However, the rate of convergence is $N^{-1/2}$, due to the central limit theorem, which is slow for high-precision computations. To obtain a precision of $4$ decimal digits with high probability then requires $10^{8}$ samples. 

The Monte Carlo method is based on a very physical understanding of the volume of a set and is a probabilistic method. Instead, the volume can be computed using a deterministic approach rooted in an algebraic understanding of volumes as integrals. Consider for example the $\ell_4$-ball $C$ in $\mathbb R^3$. Its volume can be given in closed form using the $\Gamma$-function: 
\begin{align*}
    \vol(C) &=\frac{( 2 \Gamma(1 + 1/4))^3}{\Gamma(1 + 3/4)} \\&
= 6.481987351786382022151846056460487...
\end{align*}
This number is a \textit{period}, which means it is the value of a definite integral.
Another classical volume computation method uses the Strzebonski approach \cite{Strzebonski_volumes} to Collin’s Cylindrical Algebraic Decomposition \cite{Collins-CAD} (CAD) to decompose the integration contour into cells defined by algebraic functions. The volume can then be determined by providing the CAD as input to symbolic or numerical integration methods as implemented in Mathematica~\cite{Mathematica14}. This approach may fail in the case of complicated polynomial input, see Example \ref{eg:mathematica}. Moreover, in some cases, the approach presented in this paper outperforms the CAD approach, see Example~\ref{eg:mathematica2}.

We are interested in computing periods without using symbolic integration but instead relying on the theory of holonomic functions. 
The exact computation of volumes of compact semi-algebraic sets with arbitrary precision was addressed by Lairez, Mezzarobba, and Safey El Din in \cite{Lairez_volumes_semi_algebraic_sets}. They realize the volumes of semi-algebraic sets as periods of rational integrals \cite{Lairez-periods}. Moreover, the volume can then be computed up to arbitrary precision by numerically solving a corresponding univariate linear differential equation, called a Picard--Fuchs equation. Illustrations of this technique can be found in \cite[§ 2]{SatStu25} and it fits into the broader framework of metric algebraic geometry \cite[Ch. 14]{Breiding_Kohn_Sturmfels_2024}.

We apply this to compute volumes of semi-algebraic convex bodies, and in this setting we provide improvements to the algorithm in \cite{Lairez_volumes_semi_algebraic_sets}. We reduce the required steps by a factor that is exponential in the dimension by focusing only on a single interval given by two critical values of a projection. We do this by introducing a new method to select the relevant critical values. 

We provide an implementation~\cite{implementation} of the algorithm. It is written primarily in  SageMath~\cite{sagemath} and uses the \texttt{ore\_algebra} package by Kauers, Jaroschek and Johansson \cite{ore_algebra}. We also incorporate additional software, namely msolve~\cite{msolve}, and the Julia~\cite{bezanson2017julia} package {HypersurfaceRegions.jl} based on~\cite{reinke2024hypersurfacearrangementsgenerichypersurfaces}.


Our article has the following outline. In Section~\ref{sec:background}, we review the necessary results from the theory of holonomic functions that allow us to view the volume of $C$ as the analytic continuation of volumes of deformed sets, each of whose volumes can be computed as a definite integral. The computation is done by solving corresponding differential equations. We elaborate on the algorithm applied to our class of semi-algebraic convex bodies. In Section~\ref{sec:adaptation}, we focus on the convexity of this class, which is closed under deformations and slices of deformations. This property guarantees the existence of only two relevant critical values for any projection in the recursive algorithm. We discuss our methods to compute these relevant critical values. Finally, in Section~\ref{sec:performance}, we elaborate on the reduction in the number of recursive calls in comparison to the general algorithm in \cite{Lairez_volumes_semi_algebraic_sets}. We compute various volumes in dimension $2$, $3$ and $4$, and discuss the Picard--Fuchs operators that arise in our computations. We also discuss questions that arise from these examples and opportunities for improvements of the algorithm.

\begin{acks}
We thank the anonymous reviewers for their careful reading and valuable feedback. We are grateful to Eric Pichon-Pharabod for his help in understanding the software systems involved. We thank Leonie Kayser and Mohab Safey El Din for useful discussions.
We are also grateful to Anna-Laura Sattelberger for her continuous support. And finally we thank Bernd Sturmfels for encouraging us to pursue this project. 
\end{acks}

\section{Volumes via differential equations}\label{sec:background}

\paragraph{Notation.} We denote by $D_{t, \textbf{x}}$ (and $R_{t, \textbf{x}}$) the Weyl algebra (and rational Weyl algebra) of linear differential operators with coefficients in $\mathbb Q[t,x_1,..., x_n]$ (and $\mathbb Q(t,x_1,..., x_n)$). 
We denote the action of a linear differential operator $P \in D_{t,\mathbf{x}}$ on a differentiable function $\varphi$ by $P \bullet \varphi$. The absence of a bullet indicates multiplication.
The notation $\operatorname{Ann}_{D_{t,\mathbf{x}}}(\varphi)$  refers to the left-ideal in $D_{t,\mathbf{x}}$ of linear differential operators $P$ such that $P\bullet \varphi = 0$.  If $\frac{p}{q}$ is the leading coefficient of $P$, then $\deg(P) := \deg(p) - \deg(q)$. We set $\mathbf{x} =(x_1, \dots, x_n)$. 
\\ 

The fundamental starting point for computing volumes using differential operators is a shift in perspective. While, of course, the volume $\vol(C)$ can be realized as an integral in $\mathbb{R}^n$ of the constant function $1$ over $C$ itself, it is much more useful to describe it as an integral over a closed integration contour in $\mathbb{C}^n$ as follows.

\begin{lemma} 
Let $C = \{ x \in \mathbb{R}^n \mid f(x) > 0 \}$ be a bounded region of $\mathbb R^n$ given by a single polynomial $f\in \mathbb{Q}[x_1,\ldots, x_n]$ and assume that its vanishing set $V(f)$ is a smooth variety. 
Then, the volume of the semi-algebraic set $C$ is a period of the rational function 
\begin{align}
    A(x) = \frac{(\partial_{x_1}\bullet f)x_1}{f}.
\end{align}
That is,
    \begin{equation}
        \vol (C) = \frac{1}{2\pi i}\int_{\Gamma} A(x) dx_1\wedge\ldots \wedge dx_n
    \end{equation}
where $\Gamma$ is a closed cycle $\Gamma \subset \mathbb{C}^n-\partial C$.
\end{lemma}

The proof is based on Stokes' theorem and Leray's residue theorem \cite[III, Thm. 2.4]{PhamIntegrals}, a higher-dimensional generalization of Cauchy's integral theorem, which explains the factor of $1/(2\pi i) $. The explicit shape of $\Gamma$, which is the Leray coboundary of $\partial C$, is not required to be known. Only that the contour is closed becomes relevant below. 
Also, the choice of $x_1$ is arbitrary since for a different $x_i$, the integrand is a cohomologous differential form on the complement of $V(f)$. 
 
This distinct viewpoint becomes powerful when the periods depend on an additional parameter $t$, so that they not only describe individual values, such as $\vol(C)$, but functions $\varphi(t)$ such that $\vol(C)=\varphi(t_{\val})$ for some value $t_{\val} \in \mathbb{Q}$.

\begin{definition}
    For an open set $U \subset \mathbb{C}$, a function 
    $\varphi :U\to \mathbb{C}$ 
    is a {\em {period of a rational function dependent on $t$}} if there exists a rational function $A \in \mathbb{C}(t,x_1,\ldots, x_n)$ such that for every $p \in U$, there exists a neighborhood of $p$ where $\varphi(t)$ can be written as 
\begin{equation}\label{eq:vol}
    \varphi(t) = \int_{\Gamma} A(t,x) dx_1\wedge \ldots \wedge dx_n 
\end{equation}
    for a closed cycle $\Gamma$ that is independent of $t$ and lies in the complement of the poles of $A(t,x)$.
\end{definition}

We will see below in Proposition~\ref{prop:slice_volumes_as_periods} that volumes of semi-algebraic sets can be represented by periods of rational functions depending on a single parameter. Theorem~\ref{thm:takayama} then implies that the exact volume can be obtained by numerically solving a linear differential equation
\begin{equation*}
    p_n(t)\frac{d^n\varphi}{dt^n}(t) + \cdots + p_1(t)\frac{d\varphi}{dt}(t) + p_0(t) \varphi(t)= 0 
\end{equation*}
with $p_0,\ldots, p_n \in \mathbb{C}[t]$ for a specific solution $\varphi(t)$ up to arbitrary precision. Note that this linear differential equation can be equivalently written as 
\begin{equation*}
    P \bullet \varphi(t)= 0 \qquad \text{where} \qquad P = \sum p_i\partial^i_t \in D_t .
\end{equation*}

To explain the existence of a non-zero linear differential operator $P$ that annihilates the periods of a rational function, we recall for the reader's convenience the necessary parts from the theory of holonomic functions and $D$-ideals.

A function in a single variable $x_1$ is {\em holonomic} if there exists a non-zero linear differential operator $P \in D_{x_1}$ that annihilates it. A function $f$ in $n$ variables is {\em holonomic} if there exists a holonomic left ideal $I$ \cite[Def. 1.4.8]{SST} in the Weyl algebra $D_\mathbf{x}$ such that all operators $P \in I$ annihilate~$f$. We refer to such an ideal as an {\em annihilating ideal} for~$f$. In practice, to show that a function $f$ is holonomic, it suffices to provide an ideal $I \subset \operatorname{Ann}_{R_\mathbf{x}}(f) \subset R_\mathbf{x}$ of finite holonomic rank since its Weyl closure $D_\mathbf{x} \cap I$ will be a holonomic annihilating ideal in $D_\mathbf{x}$ \cite[Thm. 1.4.15]{SST}. For more background on holonomic functions, we refer the reader to \cite{SatStu25}.

\begin{example}[Rational functions are holonomic]\label{eg:rational-holonomic}
    A rational function $A \in \mathbb{Q}(x_1,\ldots, x_n)$ is holonomic since it is annihilated by the operator $A \partial _i - \partial_i \bullet A $ for all $i$. The annihilating ideal 
    \begin{equation}\label{eg:holonomic}
        I = \langle A \partial _i - \partial_i \bullet A \rangle \subset R_{\mathbf{x}}
    \end{equation}
    has holonomic rank $1$. Hence, its Weyl closure 
    \begin{equation}
        J := D_{\mathbf{x}} \cap I
    \end{equation}
    is a holonomic $D_\mathbf{x}$-ideal annihilating $A$.
\end{example}

Fundamental operations on functions, such as restrictions and integration, have analogous operations at the level of their annihilating $D$-ideals. For integration, the main theoretical statement is the following. It is presented in similar form in {\cite[Thm. 5.5.1]{SST}}.
\begin{theorem}\label{thm:takayama}
    Let $J \subset D_{t, \mathbf{x}}$ be an annihilating ideal of a rational function $A\in \mathbb Q (t, \mathbf{x})$. Then, the ideal 
    \begin{align}
        \mathcal I_{t}(J): = (J + \partial_{x_1}D_{t, \textbf{x}} + \cdots +\partial_{x_n} D_{t,\textbf{x}} ) \quad \cap  \quad  D_t
    \end{align}
    annihilates periods of $A$ that depend on the parameter $t$. 
\end{theorem}
If $J$ is holonomic, then also its integration ideal $\mathcal I_{t}(J)$ is holonomic by \cite[Thm. 6.10.3]{Takayama2013} and hence non-trivial. It follows then that periods of rational functions depending on a parameter are holonomic.
In our case, recall \eqref{eq:vol}, $A(t,  \mathbf{x})$ is a holonomic function, and therefore $J$ as in Example~\ref{eg:rational-holonomic} is holonomic. Hence, also the integration ideal $\mathcal I_{t}(J)$ is holonomic. This implies that there is a non-zero univariate linear differential operator $P \in D_{t}$ and operators $Q_i \in D_{t,\mathbf{x}}$, such that \begin{align}\label{eq:picard--fuchs}
    P - \partial_{x_1}Q_{x_1} - \cdots - \partial_{x_n} Q_{x_n} \in \operatorname{Ann}_{D_{t, \mathbf{x}}}(A).
\end{align}
It follows from this presentation that $P$ annihilates the definite integral in \eqref{eq:vol}.
\begin{proof}[Proof of Theorem~\ref{thm:takayama}]
    Suppose $P$ and the $Q_1, \ldots, Q_n$ are given as above. Then,
    \begin{equation*}
        0 = \int_\Gamma (P  -\partial_{x_1}Q_{x_1} - \cdots - \partial_{x_n} Q_{x_n} )\bullet A(t,x) dx_1\wedge \ldots \wedge dx_n.
    \end{equation*}
    Since $P$ does not depend on the integration variables it can be taken out of the integral. Each of the remaining terms vanishes by Stokes' theorem. It is applicable, since $Q_{x_i}\bullet A$ defines a smooth function on the complement of the poles of $A$ since $Q_{x_i} \in D_{t,\mathbf{x}}$.
\end{proof}

The process of computing operators of the form \eqref{eq:picard--fuchs} is called {\em{creative telescoping}} \cite{ZEILBERGER1990321}. The operator $P$ is called the {\em telescoper} and the operators $Q_i$ are called the {\em certificates.} In practice, creative telescoping is often done iteratively by integrating out a single variable at a time. That is, given $J$, at each step one computes
\begin{equation}
    \mathcal I_{t, x_1, \dots, x_i}(J) = (\mathcal I_{t,x_1, \dots, x_{i+1}}(J) + \partial_{x_{i+1}}D_{t, x_1, \dots, x_{i+1}}) \cap D_{t, x_1, \dots, x_{i}}.
\end{equation}
In each iterative step, telescopers and their respective certificates are computed. 
We refer the reader to \cite[Sec. 5.4]{KauersDfinite}, \cite{chyzak:tel-01069831}, and most recently \cite{brochet2025fastermultivariateintegrationdmodules} for details on the various available creative telescoping algorithms.  
The final operator $P$ in \eqref{eq:picard--fuchs} is called a {\em{Picard--Fuchs operator}}, a non-zero linear differential operator in $D_t$ that annihilates periods of $A \in \mathbb{Q}(t,\mathbf{x})$ depending on $t$.

Now consider concave polynomials $f_1, \dots, f_k \in \mathbb Q[x_1, \dots, x_n]$. They define the convex, compact semi-algebraic set
\begin{equation}
    C  = \bigcap_{i=1}^k\{x \in \mathbb R^n \mid f_i(x) >0\}.
\end{equation}
Moreover, $C$ can be described as the limit (in the Hausdorff metric) of the $1$-parameter family of semi-algebraic sets
\begin{equation}\label{eq:deformed-intersection}
    C_t = \{x \in \mathbb R^n \mid \prod_i f_i(x) -t > 0 \} \cap C. 
\end{equation}
Since the deformed product
\begin{equation}\label{eq:deformed_product}
    F_t := f_1\cdot\ldots \cdot f_k -t \in \mathbb{Q}[t, \mathbf{x}]
\end{equation}
has nowhere vanishing Jacobian, it follows that its vanishing set $V(F_t) \subset \mathbb{R}^{n+1}$ is a  smooth variety. By Sard's theorem, also the slice $V(F_t) \cap \{t\}\times \mathbb{R}^n$ is smooth for all but finitely many values of $t$. This then also holds for any connected component of $V(F_t)$, such as the boundary of $C_t$ which we denote by $\partial C_t$. We will refer to $C_t$ as the (smooth) {\em deformation} of $C$. The following result is a direct application of Proposition~\ref{prop:slice_vol} below. 
\begin{proposition}
    \label{prop:slice_volumes_as_periods}
    Let $C_t$ be defined as in \eqref{eq:deformed-intersection}. Then 
    \begin{equation}
    \varphi : (0, \varepsilon) \to \mathbb{R}, \quad t\mapsto\vol(C_t)
\end{equation}
is a period of a rational integral depending on the parameter $t$ for small $\varepsilon>0$. Moreover, $\vol(C)$ is the limit as $t$ tends to $0$ of the analytic continuation of $\varphi(t)$.
\end{proposition}

This shows that the volume function $\varphi(t)$ is a solution of a Picard--Fuchs operator $P_t$ in $D_t$. In the above, $\varepsilon$ can be chosen to be the smallest positive singular value of $t$ for $P_t$. Locally, on simply connected regions away from the singular locus of $P_t$, the solutions of $P_t$ form a $\mathbb C$-vector space of dimension $\ord(P_t)$. Thus, our particular solution $\varphi(t)$ can be specified in this vector space by providing suitable initial conditions. By \cite[Lemma 15]{Lairez_volumes_semi_algebraic_sets}, the initial conditions can be of the following form:
\begin{lemma}\label{lemma:volct}
    The solution of $P_t$, which realizes $\varphi(t) = \vol(C_t)$ on an open interval $(0,\varepsilon)$, can be uniquely determined by providing the value of $\vol(C_t)$ at $\ord(P_t)$ many suitable points $t \in (0,\varepsilon)$.
\end{lemma}
Once $\varphi$ is determined within the solution space of $P_t$ on $(0,\varepsilon)$, the volume of $C$ is then obtained by analytically continuing $\varphi$ to $t=0$. The word "suitable" in Lemma \ref{lemma:volct} has to be understood as follows. Along a path $\gamma: t_0 \to t_1$ outside the singular locus of a linear differential operator $P$, analytic continuation provides an isomorphism of the $\mathbb C$-vector space of solutions $\operatorname{Sol}_{t_0}$ at $t_0$, and $\operatorname{Sol}_{t_1}$ at $t_1$. 
This isomorphism can be described in a given basis and can be computed numerically, for example using high-precision solvers such as implemented in the \texttt{ore\_algebra} package \cite{Mezzarobba2016} up to arbitrary precision. Providing the value $\vol(C_t)$ at a point $t_i$ fixes only a single coordinate in $\operatorname{Sol}_{t_i}$. The $\ord(P_t)$ many values for $t$ are then suitable if together they determine, namely as a linear system, all coordinates in $\operatorname{Sol}_{t_0}$ uniquely.

Let us now describe how the values of $\vol(C_t)$ can be determined recursively as volumes of lower-dimensional semi-algebraic sets. 
By a {\em slice} of $C_t$, we mean the intersection of $C_t$ with a hyperplane. 
Since the boundary of a slice of a smooth set is also smooth, a more general statement can be made about slices of $C_t$ and their volumes.
\begin{proposition}[{\cite[Thm. 9]{Lairez_volumes_semi_algebraic_sets}}]\label{prop:slice_vol}
    If $f \in \mathbb{Q}[x_1,\ldots,x_n]$ such that the boundary of $\mathcal C = \{x \in \mathbb R^n \mid f(x)>0\}$ is smooth, then for $C$ a union of bounded, connected components of $\mathcal C$,
    the volume of the slice 
    \begin{equation*}
        \varphi(v) = \vol(C \cap \{x \in \mathbb R^n \mid x_i=v\})
    \end{equation*}is a period of the rational function
\begin{equation}\label{eq:rational_function}
    A = \frac{(\partial_{x_j}\bullet f(x_i = v))x_j}{f(x_i = v)}, \quad \text{for any}\quad i\not = j,
\end{equation}
on any open interval of adjacent critical values of the projection
\begin{equation}\label{eq:proj}
    \pr_{x_i} : \partial C \rightarrow \mathbb{R}, \qquad \mathbf{x} \mapsto x_i.
\end{equation}
\end{proposition}

Let $(c_1,c_2) \subset \mathbb{R}$ be such an interval of adjacent critical values for $C_t$. The volume of $C_t$ over this interval can be computed by evaluating the integral
\begin{equation}
    \psi(s) = \int^{s}_{c_1}\vol(C_t \cap \{x \in \mathbb R^n \mid x_i=v\})dv
\end{equation}
at $s=c_2$. By Proposition~\ref{prop:slice_vol}, the integrand is annihilated by a Picard--Fuchs operator $P \in D_{x_i}$. Therefore, by the fundamental theorem of calculus, it follows that
\begin{equation}
    P\partial_{x_i} \bullet \psi = 0.
\end{equation} 
As before for $\varphi$, one can solve for $\psi$ by providing suitable initial conditions, this time in terms of the values of $\psi'(s)$, which is the volume of the slice over $x_i =s$.

This leads to an algorithm that computes the volume integral \eqref{eq:vol} as a solution of a Picard--Fuchs operator. Since solving the Picard--Fuchs operator requires initial conditions which are volume computations of lower-dimensional slices, the algorithm is recursive with depth $n$. Algorithm~\ref{alg:volume2} takes as input $k$ concave polynomials and returns the volume of $C$ to a chosen precision. It implicitly calls Algorithm~\ref{alg:volume1}, a recursive algorithm that computes the volume of the deformed convex body. Algorithms~\ref{alg:volume2} and~\ref{alg:volume1} were presented in more general form in \cite{Lairez_volumes_semi_algebraic_sets}.
The two algorithms make use of several subroutines, which will only be shortly described in this section. Subroutines (2), (4) and (5) are explained further in Section~\ref{sec:adaptation}.
\begin{enumerate}
    \item \texttt{CreativeTelescoping}$(I, x_i)$: Returns an element $P$ of the integration ideal of $I$ that integrates out all variables but $x_i$, using a suitable creative telescoping algorithm. See also Section~\ref{sec:performance}.
\end{enumerate}
\begin{enumerate}
    \setcounter{enumi}{1}
    \item \texttt{SuitableValues}$(P, (a,b))$: Returns a list of suitable values for $x_i \in (a,b)$ to uniquely determine a solution of $P$ on $(a,b)$. See also Section \ref{sec:adaptation}. 
\end{enumerate}
\begin{enumerate}
    \setcounter{enumi}{2}
    \item \texttt{Solve}($P$, $L_{ic},x_i = v$): Determines the solution $\varphi$ of $P$ by the list of initial conditions $L_{ic}$ and returns its value $\varphi(x_i = v)$ for a value $v \in [a,b]$.
\end{enumerate}
 \begin{enumerate}
    \setcounter{enumi}{3}
     \item $\texttt{CriticalValues}((f_1,\ldots, f_k), t_\val, x_{\pr})$: Returns the critical values of the projection from $\partial C_{t_\val}$ onto the $x_\pr$ axis. 
     We will prove in Proposition~\ref{prop:deformed-intersection} that for the class of concave polynomials, the deformation $C_{t_\val}$ is convex. Hence, there will only be two such critical values.
     How to obtain them is described in Algorithm~\ref{alg:critical_values}.
 \end{enumerate}
 \begin{enumerate}
    \setcounter{enumi}{4}
     \item $\texttt{1DimVolume}(L_{\operatorname{res}}, t_\val, N)$: Returns the $1$-dimensional volume of the deformed intersection, assuming that the restricted polynomials in $L_{\operatorname{res}}$ are univariate, see also Section~\ref{sec:adaptation}.
 \end{enumerate}
 \begin{enumerate}
    \setcounter{enumi}{5}
     \item $\texttt{ProjectionVariable}((f_1,\ldots,f_k), t_{\val})$: Selects one of the available variables to project onto next. This becomes relevant in Section~\ref{sec:performance}.
 \end{enumerate}

\begin{algorithm}[H]
    \caption{\texttt{VolumeSemialgebraic}} \label{alg:volume2}
    \begin{algorithmic} 
    \smallskip
        \Require Concave polynomials $f_1, \ldots, f_k \in \mathbb Q [x_1, \dots, x_n]$, precision $N \in \mathbb N$.
        \Ensure Volume of the convex body $C$ up to $N$ binary digits.
        \State $F_t:= \prod_i f_i - t$
        \State $A_t := \frac{(\partial_{1}\bullet F_t)x_1}{F_t}$
        \State $I_t := \langle A_t \partial_\alpha - \partial_\alpha \bullet A_t \mid \alpha \in \{x_1,...,x_n,t\} \rangle$ 
        \State $P_t :=  \texttt{CreativeTelescoping}(I_t, t)$
        \State $\varepsilon := \min \{|t| \mid t \in \texttt{SingLoc}(P_t)\setminus\{0\}\}$
        \State $L_t :=\texttt{SuitableValues}(P_t, (0,\varepsilon))$
        \State $L_{ic}:=$ empty \texttt{list} of \texttt{InitialConditions} 
        \For{$t_{\val} \in L_t$}
            \State $\texttt{append}(L_{ic}, \texttt{SmoothVolume}((f_1,\ldots, f_k), t_{\val}, N))$
        \EndFor
        \State $\operatorname{vol_0} := \texttt{Solve}(P_t, L_{ic},t = 0)$
        \State \Return $\operatorname{vol_0}$
    \end{algorithmic}
\end{algorithm}

\begin{algorithm}[ht]
    \caption{\texttt{SmoothVolume}} \label{alg:volume1}
    \begin{algorithmic} 
    \smallskip
        \Require Concave polynomials $(f_1,\ldots, f_k)$ in $\mathbb{Q}[x_1,\ldots, x_n]$, deformation value $t_{\val} \in \mathbb Q$, number of precision bits $N \in \mathbb N$.
        \Ensure Volume of the deformation $C_{t_\val}$ up to $N$ binary digits.
        \State $\operatorname{x_{\pr}} := \texttt{ProjectionVariable}((f_1,\ldots,f_k), t_{\val})$
        \State $F_{t_{\val}} := \prod f_i - t_{\val}, \quad A := \frac{(\partial_{x_i}\bullet F_{t_{\val}})x_i}{F_{t_{\val}}}$ for $x_i \neq x_\pr$ 

        \State $I := \langle A \partial_i - \partial_i \bullet A \mid 1 \leq i \leq n \rangle$
        \State $P := \texttt{CreativeTelescoping}(I, x_{\pr})$
        \State $c_1,c_2 := \texttt{CriticalValues}((f_1,\ldots, f_k), t_\val, x_{\pr})$
        \State $L := \texttt{SuitableValues}(P, (c_1,c_2))$
        \State $L_{ic}:=$ empty \texttt{list} of \texttt{InitialConditions} 
        \For{$x_{\val} \in L$}
            \State $L_{\operatorname{res}}:=(f_1,\ldots, f_k)(x_\pr = x_{\val})$
            \If{$n>2$}
                \State $\texttt{append}(L_{ic}, \varphi'(x_{\val})=\texttt{SmoothVolume}(L_{ \operatorname{res}}, t_{\val},N))$
            \Else $(n=2)$
                \State append$(L_{ic}, \varphi'(x_{\val}) = \texttt{1DimVolume}(L_{\operatorname{res}}, t_\val, N))$
            \EndIf
        \EndFor  
        \State \Return 
        \texttt{Solve}($P\partial_{x_\pr}$, $L_{ic},x_\pr = c_2$)
    \end{algorithmic}
\end{algorithm}

\section{Convex Sets and Their Deformations}\label{sec:adaptation}

A polynomial $f \in \mathbb{Q}[\mathbf{x}]$ is a \textit{concave function} on a convex set $U \subseteq \mathbb R^n$ if for any $x, y \in U$ and $\alpha \in [0,1]$
\begin{equation}\label{eqdef:concave}
    f\big(\alpha x + (1 - \alpha)y\big) \geq \alpha f(x) + (1-\alpha)f(y).
\end{equation}

\begin{lemma}\label{lemma:superlevelsets}
    Let $U \subset \mathbb{R}^n$ be a convex region and $f: U \rightarrow \mathbb{R}$ a concave function on that region. Then for any $t$, the super-levelset
    \begin{equation}
        f_{>t} := \{x \in \mathbb{R}^n \mid f(x) > t\}
    \end{equation}
    is a convex subset of $\mathbb{R}^n$.
\end{lemma}
\begin{proof}
    This is a consequence of the definition of concavity. If for $x,y \in U$, both $f(x)>t$ and $f(y)>t$, then this holds by definition for any point along the line segment connecting $x$ and $y$.
\end{proof}
Let now $f_1,\ldots, f_k \in \mathbb{Q}[\mathbf{x}]$ be a list of concave polynomials. Then their common positivity locus $C$ is a convex subset of $\mathbb{R}^n$, where
\begin{equation}\label{eq:setupC}
    C = \bigcap^k_{i=1}\{x\in \mathbb R^n \mid f_i(x) > 0\}.
\end{equation}

\begin{example} 
    Fix an even integer $p \in 2\mathbb{N}_{>0}$. 
    Consider the unit $\ell_p$-ball in $\mathbb R^n$ centered at $\mu$. 
    It is a semi-algebraic set given by the concave polynomial
\begin{equation}
    f_{\ell_p, \mu} = 1- ((x_1-\mu_1)^p + \cdots + (x_n-\mu_n)^p).
\end{equation}
    We denote the corresponding translated $\ell_p$-ball by $C_{\ell_p,\mu} := {(f_{\ell_p, \mu})}_{>0}$.
    Given $\mu_1$ and $\mu_2$, if $C:=C_{\ell_p,\mu_1} \cap C_{\ell_p, \mu_2} \neq \emptyset$, 
    then $C$ is one out of two connected components of
    $({f_{\ell_p, \mu_1}f_{\ell_p, \mu_2}})_{> 0}$.
    Figure~\ref{fig:deformation} shows the deformation of two $\ell_4$-balls in $\mathbb R^2$ centered at $(0,0)$ and $(1/2, 1/3)$. The deformed product $C_t \subsetneq (F_t)_{> 0}$ is then defined by \begin{gather}
        F_t = (1 - x^4 - y^4) (1 - (x-1/2)^4 - (y-1/3)^4)) - t.
    \end{gather}
    Figure~\ref{fig:deformation}(a) shows $V(F_t(t=0))$ and \ref{fig:deformation}(b) shows $V(F_t(t = 0.3))$.

    \begin{figure}%
    \centering
    \subfloat[\centering $t = 0$]{{\includegraphics[width=4cm]{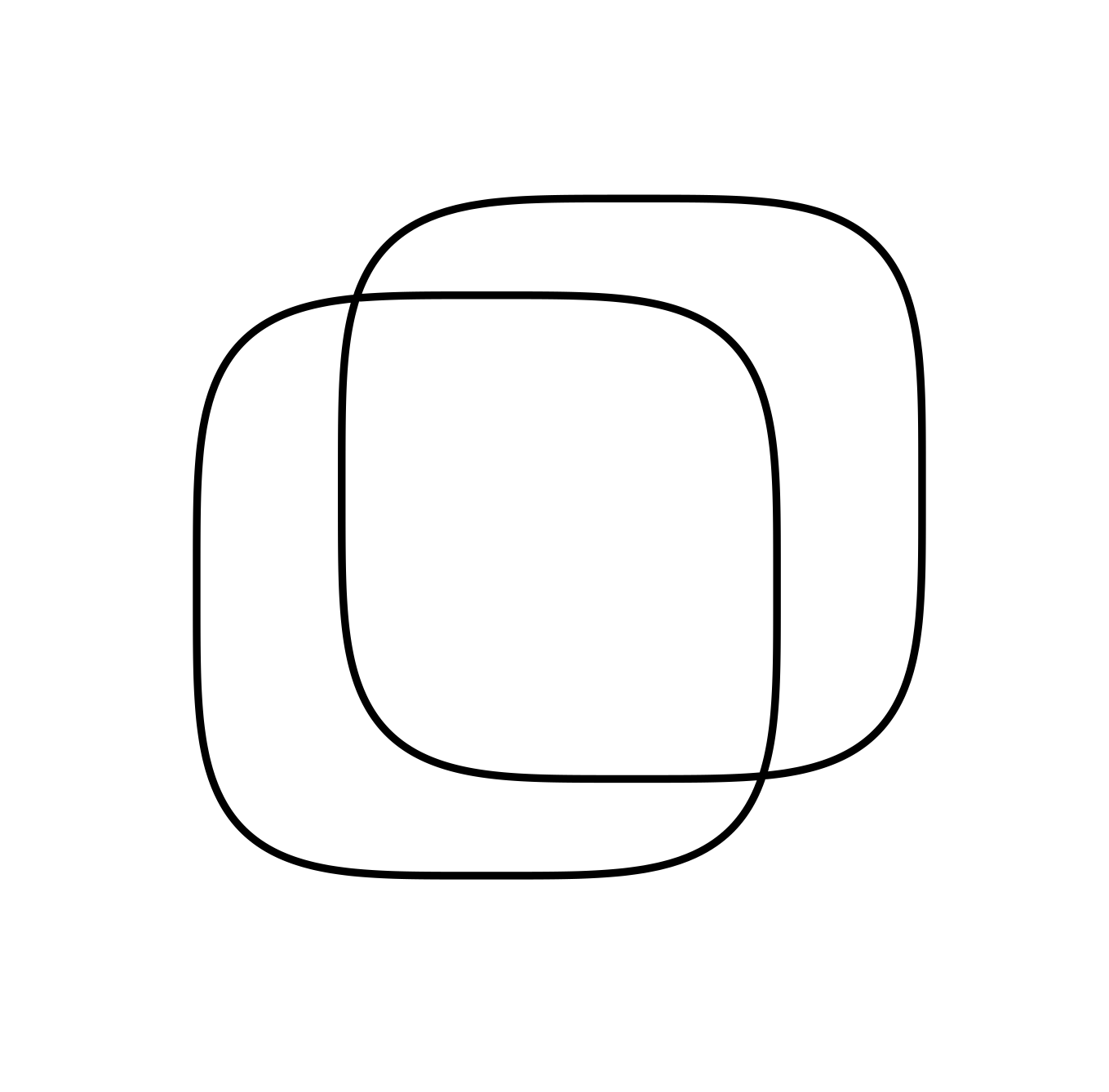} }}%
    \quad
    \subfloat[\centering $t = 0.3$]{{\includegraphics[width=4cm]{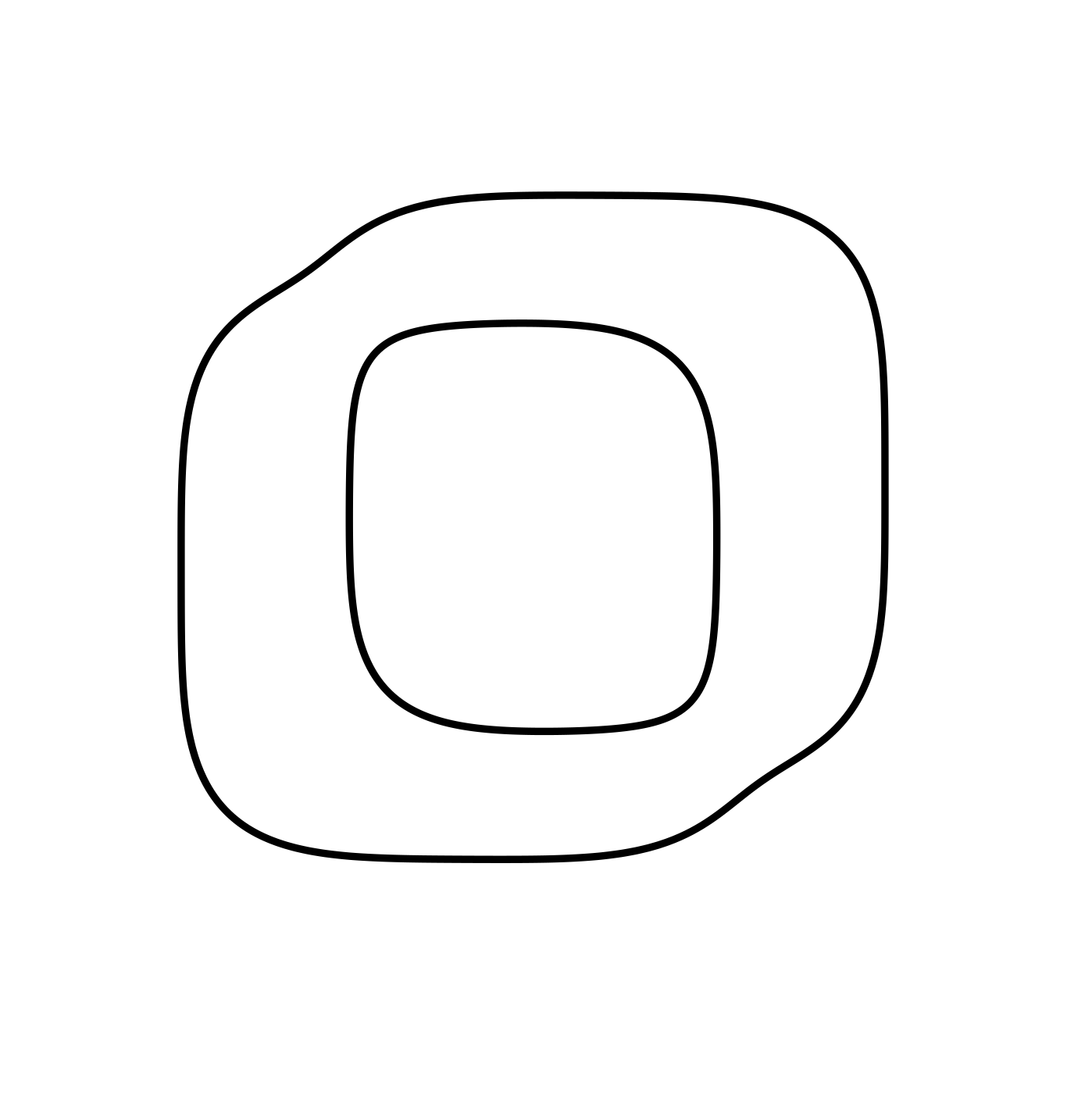} }}%
    \caption{Two $\ell_4$-balls deformed by a parameter $t$}%
    \label{fig:deformation}%
\end{figure}
\end{example}
Let us now study the deformed intersection $C_t$. 
For general polynomials, even though for small $t$ the deformation $C_t$ is connected, its slices
may have multiple components, as in the following example.
\begin{example} 
Consider the semi-algebraic set $C \subset \mathbb{R}^2$ defined by the common positivity locus of
\begin{equation*}
    f := -(y^2-(x+1)x^2 -\frac{1}{2})
\end{equation*}
and the three affine linear polynomials
\begin{equation*}
    l_1 := -(x - 2), \quad l_2 := (y + \frac{1}{2}), \quad l_3 := -(y - \frac{1}{2}).
\end{equation*}
In this case, the semi-algebraic set $C$, depicted in gray in Figure \ref{fig:guitar}(a), is convex. However, its deformation $C_t$, seen in gray in Figure \ref{fig:guitar}(b) for $t=0.2$ is not. In particular, there exists some value $v$ for which the slice $C_t \cap \{y =v\}$ has two connected components.
\begin{figure}%
    \centering
    \subfloat[\centering $t = 0$]{{\includegraphics[width=4cm]{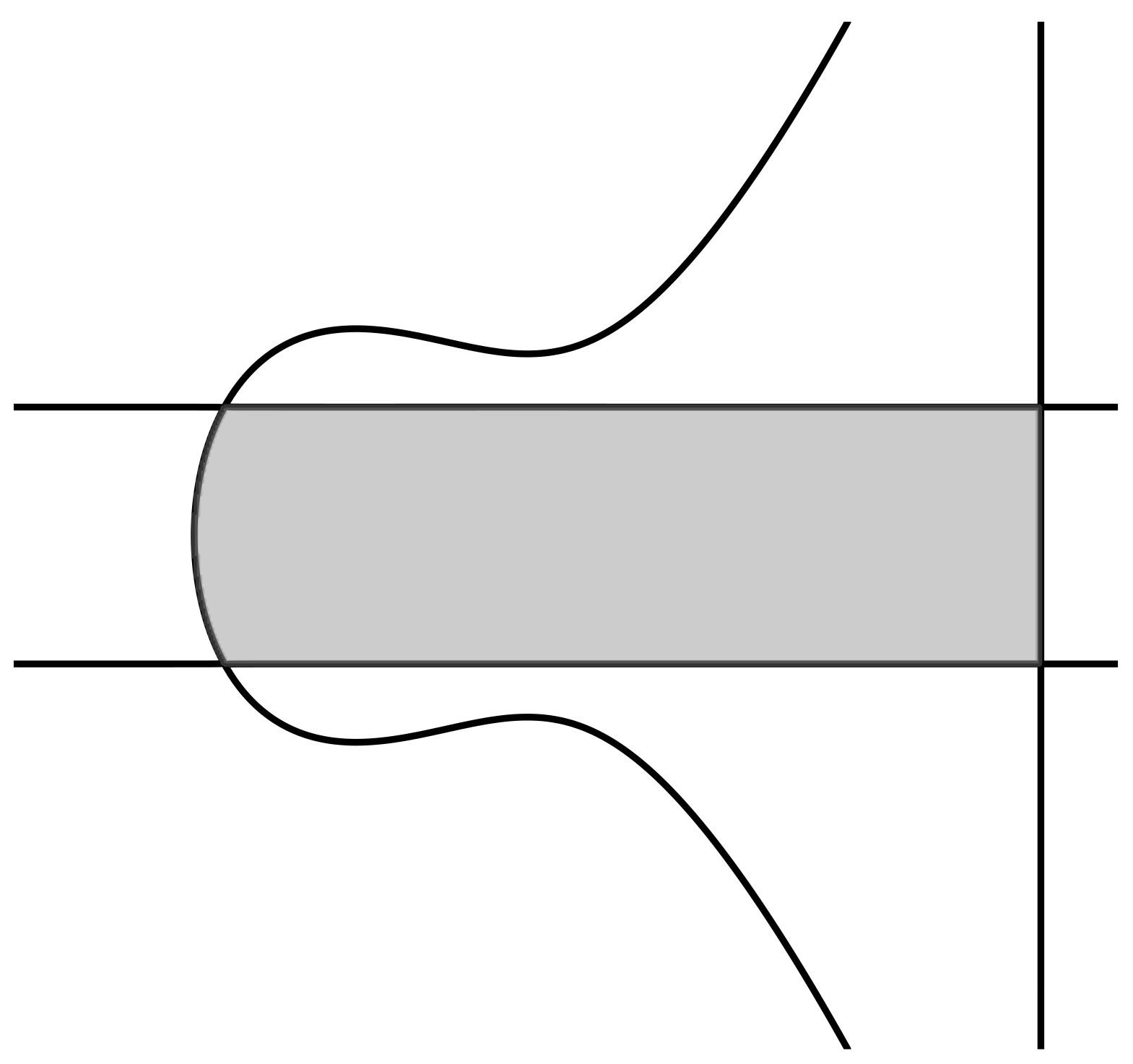} }}%
    \;
    \subfloat[\centering $t = 0.2$]{{\includegraphics[width=4cm]{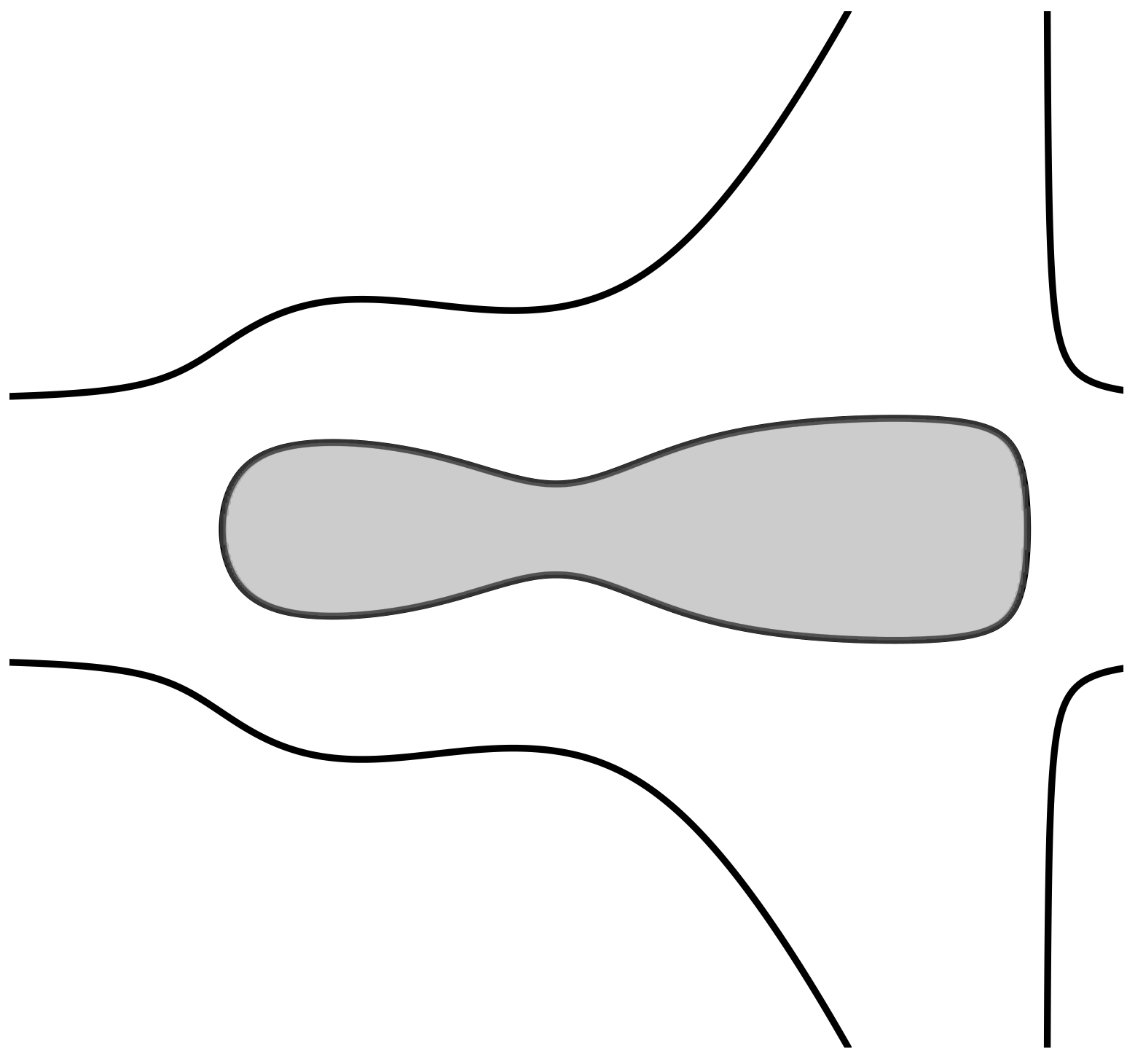} }}%
    \caption{Convex set with non-convex deformation}%
    \label{fig:guitar}%
\end{figure}
\end{example}
For concave polynomials, we can guarantee that the deformation $C_t$, and hence all its slices, are convex for all $t$.
\begin{proposition}\label{prop:deformed-intersection}
    Let $C$ be the common positivity locus of concave polynomials $f_1,\ldots,f_k \in \mathbb{Q}[\mathbf x]$. Then,
    its deformation $C_t$ as defined in \eqref{eq:deformed-intersection} is a convex set for all $t>0$.
\end{proposition}
\begin{proof}
    Recall that a positive function $f$ is defined to be log-concave if $\log(f)$ is also concave. Observe then that Lemma~\ref{lemma:superlevelsets} also holds for log-concave functions since the logarithm is a strictly monotonic function.
    
    In our case, the deformations $C_t$ are the super-levelsets of the product $F:=\prod^k_{i=1} f_i$ when restricted to $C$. But $C$ is the common positivity locus of all $f_i$, where automatically each of the $f_i$ is log-concave.
    Then, log-concavity of $F$ follows since 
    \begin{equation*}
        \log(F) = \log(f_1) + \ldots + \log(f_k)
    \end{equation*}
    and since sums of concave functions are concave by the defining Equation~\eqref{eqdef:concave}.
\end{proof}

It naturally follows from this observation that for any $x_i$, the projection of the boundary of the deformed intersection $\partial C_t$ onto the $x_i$-axis has exactly two critical values, namely the minimum and maximum of $x_i$ attained on $C_t$. Moreover, Proposition~\ref{prop:deformed-intersection} and this conclusion hold for any non-empty slice obtained by transverse intersection with an affine hyperplane since concave functions are naturally preserved under restriction to affine-linear subspaces.

Consider now the projection onto the $x_i$-axis. The map
\begin{equation}
    \pr_{x_i} : \partial C_t \rightarrow \mathbb{R}, \quad \mathbf{x} \to x_i
\end{equation}
 has precisely two critical values $c_1 < c_2 \in \mathbb{R}$ and the function 
\begin{equation}
    \varphi: \mathbb{R} \rightarrow \mathbb{R}, \quad v \to \vol(C_{t, x_i = v})
\end{equation}
is a period of a rational function when $\varphi$ is restricted to the open interval $(c_1, c_2) \subset \mathbb{R}$.

Let us now describe how the implementation of the routines \texttt{CriticalValues} and \texttt{1DimVolume} profits from the setting of concave polynomials, followed by a note on \texttt{SuitableValues}.

\paragraph{\texttt{CriticalValues}}
The two critical values coming from $\partial C_t$ can
be identified among the critical values coming from the projection
of the vanishing set of $F_t$ in the following way. We suppose that $t \in \mathbb{Q}_{>0}$ is fixed, so that $F_t \in \mathbb{Q}[\mathbf{x}]$. Then,  
the locus of critical points $\critPts$ of the projection of the algebraic set $V(F_t) \subset \mathbb{R}^n$ is the vanishing set of the ideal 
\begin{align}
    I = \langle \tilde{F_t}, \partial_{x_1}\bullet \tilde{F_t}, \dots, \widehat{\partial_{x_i}\bullet \tilde{F_t}}, \ldots \partial_{x_n}\bullet \tilde{F_t} \rangle
\end{align}
where $\tilde{F_t}$ denotes the square-free part of $F_t$ and where
the derivative with respect to $x_i$ is omitted. The critical values $\critVals_i$ of the projection of $F_t$ onto the $x_i$-axis are then obtained by elimination. The two relevant critical values corresponding to the projection of $\partial C_t$ can be identified in two possible ways:

\paragraph{Case 1 $(\dim(\critPts) = 0)$:} In this case, there are only finitely many points in the locus of critical points. We check which of them lie in $C$ by evaluating each of the $f_i$'s. By the convexity of $C_t$ shown in Proposition~\ref{prop:deformed-intersection} and since $\dim(\critPts) = 0$, we know that there will be exactly two such critical points. By projecting them back onto the $x_i$-axis, we have found the two relevant critical values.

\paragraph{Case 2 $(\operatorname{dim}(\critPts)>0)$:} 
In this case, we employ the subroutine \texttt{SamplePointsHypersurfaceRegions($p_{\text{cv}}\cdot F_t$)} where $p_{\text{cv}}$ generates the elimination ideal of $I$ with respect to all variables except $x_i$. This subroutine samples one point from each region in the complement of the hypersurface arrangement 
\begin{equation}
      V(F_t) \bigcup_{s \in \critVals_i}\{x \in \mathbb R^n \mid x_i = s\}.
\end{equation} Then, we identify the sample points that lie in the deformed intersection $C_t$ and project them onto the $x_i$-axis. This set of projected values defines a closed interval $[a,b]$. Then, the points in $\critVals_i$ that are closest to, but outside, on either side of $[a,b]$ are the critical values of the projection of $C_t$ onto the $x_i$-axis. 

\begin{algorithm}[ht]
    \caption{\texttt{CriticalValues}} \label{alg:critical_values}
    \begin{algorithmic} 
    \smallskip
        \Require Concave polynomials $f_1,\ldots, f_k \in \mathbb{Q}[\mathbf{x}]$, projection variable $x_i$, deformation value $t_{\val} \in \mathbb{Q}_{>0}$.
        \Ensure Critical values $c_1,c_2$ of the projection $\pr_{x_i} : \partial C_{t_{\val}} \rightarrow \mathbb{R}$.

        \State $f$ := \texttt{SquareFree}($\prod f_i -t_{val}$)
        \State $I$ := $\langle f, \partial_1\bullet f, \ldots, \widehat{\partial_i\bullet f}, \ldots \partial_n \bullet f\rangle$
        \If{$\dim(I) = 0$}
            \State $\critPts := \texttt{msolve}(I)$
            \State $a,b = \operatorname{Pts}\, \cap\, C$
            \State $c_1 ,c_2 = \operatorname{pr_{x_i}}(a,b)$
        \Else
            \State $\langle p_{\text{cv}} \rangle$ = \texttt{Eliminate}($I$, $x_1,\ldots, \widehat{x_i}, \ldots x_n)$
            \State CritVals := \texttt{Roots}($p_{\text{cv}}$, $\mathbb{R}$)
            \State pts := \texttt{SamplePointsHypersurfaceRegions}($p_{\text{cv}}\cdot f$))
            \State $c_1$ := $\texttt{max}\{p \in \text{CritVals} \mid p < \texttt{min}(\pr_{x_i}(\text{pts }\cap C_{t_{val}}))\}$
            \State $c_2$ := $\texttt{min}\{p \in \text{CritVals} \mid p > \texttt{max}(\pr_{x_i}(\text{pts}\cap C_{t_{val}}))\}$
        \EndIf
        \State \Return $\{c_1,c_2\}$
    \end{algorithmic}
\end{algorithm}

\begin{example}[Hypersurface regions]\label{eg:hypersurfregions}
    Let us provide an example of encountering a positive-dimensional ideal when computing the critical values of a projection. Consider the polynomials \begin{align*}
        f_1 = 1 - (x^2 + y^2 + z^2 + w^2) \\
        f_2 = 1 - ((x-1)^2 + y^2 + z^2 + w^2),
    \end{align*}
    defining two Euclidean balls and their intersection $C$.
    First, 
    $C$ is deformed to $C_t$ for $t$ small, and then projected onto the $y$-axis. The choice of this axis is discussed in Section~\ref{sec:performance}.
    Then, for $y$ in a set of suitable values, the $3$-dimensional slice is projected onto the $x$-axis. 
    Here, the ideal of the critical locus is positive-dimensional.
    So, we employ \texttt{SamplePointsHypersurfaceRegions} to be able to select the correct critical values. Figure~\ref{fig:hypersurfregions} shows a 2-dimensional picture of the different regions of the hypersurface arrangement, and the points sampled from each of them. 
    \begin{figure}
        \centering
        \includegraphics[width=0.8\linewidth]{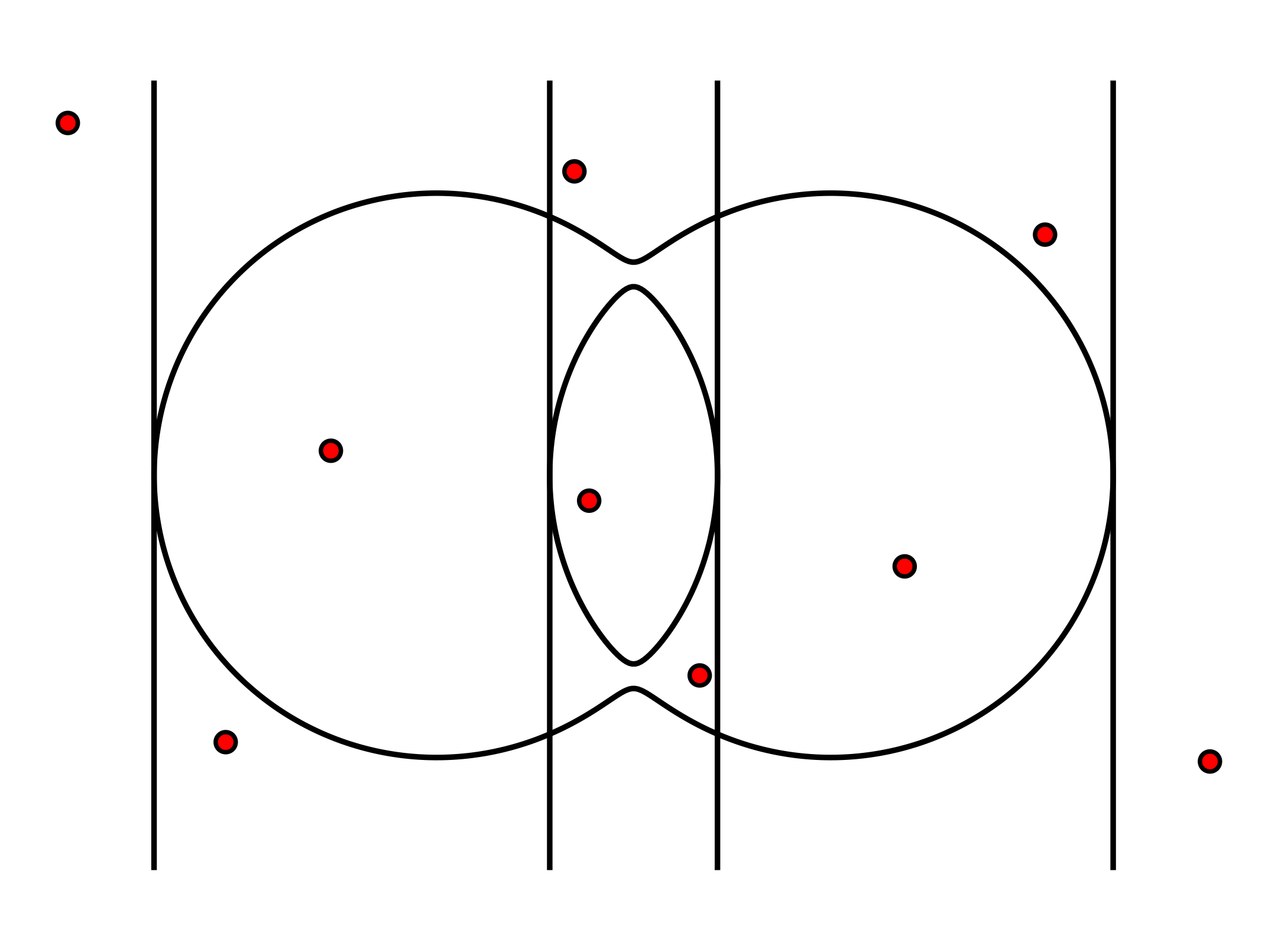}
        \caption{Sampling points to select the correct critical values}
        \label{fig:hypersurfregions}
    \end{figure}
\end{example}

\paragraph{\texttt{1DimVolume}}

If $\operatorname{dim}(C_t ) = 2$, then after computing the relevant critical values $c_1,c_2$ of the projection, the algorithm calls for computing volumes of $1$-dimensional slices at $x_i = v$ for $c_1< v< c_2$. This is done by intersecting the curve $F_t$ with the line $x_i = v$. By convexity, there are only two intersection points that lie inside $C$. Denote them by $\lambda_1$ and $\lambda_2$. Then, the required $1$-dimensional volume is $|\lambda_1 - \lambda_2 | $.

\paragraph{\texttt{SuitableValues}}

Recall the notation of ``suitable" from Lemma 2.6 and the subsequent description. Once we have correctly selected the two critical values of a projection, it remains to select the values at which we take slices.
Let $p_0$ be the smallest singular point of $P$ in the interval $(c_1,c_2)$. 
A randomly sampled set of $\ord(P)$ many values in $(c_1,p_0)$ is suitable with probability $1$, see also \cite[Section 4]{Lairez_volumes_semi_algebraic_sets}.
In practice, we sample at uniform intervals in $(c_1, p_0)$ and then check the invertibility of the induced linear system. 
Although the initial conditions are defined outside the singular locus, the volume function is analytic in the interval $(c_1,c_2)$ and thus the determined solution can be continued through any singular point of $P$ within the interval.

\section{Implementation and experiments}\label{sec:performance}

Our implementation of the above algorithms can be found at~\cite{implementation}. It is written mostly in SageMath, building on the \texttt{ore\_algebra} package to compute Picard--Fuchs operators and to solve them to arbitrary precision. Furthermore, msolve is used to compute $\critPts$ when the ideal of the critical locus is $0$-dimensional, and to find the intersection points in \texttt{1DimVolume}. Finally, when the ideal of the critical locus is positive-dimensional, we use the Julia package {HypersurfaceRegions.jl} \cite{reinke2024hypersurfacearrangementsgenerichypersurfaces} to sample points in the complement of~$V(F_{t_\val}) \cup_{s\in \operatorname{CritVals}} \{x_i = s\} $.

\vspace{0.3cm}

The most expensive step of the volume computation is the process of creative telescoping. Therefore, it is of great interest to minimize the number of times a Picard--Fuchs operator is computed. Let us consider a volume computation for a semi-algebraic set in $\mathbb R^n$. Let $d$ denote the maximum order of all Picard--Fuchs operators computed in the recursive algorithm. 
The general algorithm in \cite{Lairez_volumes_semi_algebraic_sets} would require a creative telescoping computation for every slice value (at most $d$) for every pair of adjacent critical values of a projection onto a coordinate axis, recall \eqref{eq:proj}. If $c$ denotes the maximum number of pairs of adjacent critical values of any projection of a slice, then this would mean that the number of creative telescoping steps required is bounded above by $(c \cdot d)^n$. 

In contrast, by restricting to the class of convex semi-algebraic sets defined by concave polynomials, the number of relevant critical values for any projection is exactly $2$, thus requiring only one telescoping step for each slice. Thus, the number of creative telescoping steps is bounded above by $d^n$. 
This reduces the general recursive algorithm, which traverses every node of a tree, to a recursive algorithm on a path. This improves the number of iterations from the general case by an exponential factor in the dimension $n$. 

\vspace{0.3cm}

Let us now compute exact volumes of various convex bodies. In doing so, we encounter many interesting questions, and opportunities for future work. These examples were run on a personal computer with an Apple M3 chip and $32$ GB of memory. All of these examples can be reproduced using our implementation by running the Jupyter notebook found at~\cite{implementation}.

\begin{example}[Two Euclidean balls in dimension $4$]\label{eq:l24}
    We continue Example~\ref{eg:hypersurfregions} and compute the volume of the intersection of two Euclidean balls in $4$ dimensions, centered at 
    $$\mu_1 = (0,0,0,0),\quad \mu_2 = (1,0,0,0).$$ The operator $P_t$ which annihilates $\vol(C_t)$ for small $t$ has order~$3$ and degree~$4$. The value of the slices $\vol(C_t)$ are computed for $t \in (0, 5625/10000)$, where $5625/10000$ is the smallest real point in the singular locus of $P_t$. 
    Projecting in the reverse order of variables, that is $x_3, x_2, x_1, x_0$, all subsequent ideals of critical points are zero-dimensional. Finally, the volume is computed to be 
    \begin{align*}
        1.&24934384893295779102649229344462517961832527330926\\&37279549959064283763620392536285226171895009497206\\&44758421633474540092111890318221306407347172357265\\&78794809305528585515475494602146469693898082410815\\&98529998412506559308568339206044582076143683749948\\&5954393910960456071341017606352607002 \ldots
    \end{align*}
    These are the first $288$ decimal digits of the period. This was computed in 2 minutes and 15 seconds.
    The precision of the volume can be decided by the user. We can compare this output with a volume computation using the Monte Carlo method. With $10^8$ sample points, we get the volume $1.24952304$, which already differs from the exact volume at the $4^{\text{th}}$ decimal place.
\end{example}

We can compute volumes of intersections of a range of convex bodies in reasonable time. Example~\ref{eg:volumes-many-bodies} shows truncated exact volumes of some convex bodies in dimensions $2$ and $3$, and the (average) computation times. 

\begin{example}[More volumes]\label{eg:volumes-many-bodies}
We compute the volumes of the following convex bodies up to a precision of $10^{-50}$:
\begin{enumerate}
    \item \textit{A triangle} given by $1 - x - y, \;y,\; x - y + 1 $
    \item \textit{A tetrahedron}  given by $x,\; y, \;z, \; 1/3 - x - y- z $
    \item \textit{An ellipse and a circle} given by $1 - 4x^2 - (y-1/2)^2, 1 - (x-1/2)^2 - y^2$
    \item \textit{Two $\ell_4$-balls in $\mathbb R^2$} given by $1 - x^4 - y^4, 1 - (x-1/2)^4 - y^4$
\end{enumerate}

    \begin{center}
\begin{tabular}{ |c|c|c| } 
 \hline
 \textbf{convex body} & \textbf{volume} & \textbf{time} \\ 
 \hline
 Triangle & 1.0000000000000... & 0.9s \\ 
 Tetrahedron  & 0.006172839506172839506... & 6.8s \\ 
 Ellipse and circle & 1.063610448155437831407... & 2m 15.0s \\
 Two $\ell_4$-balls in $\mathbb R^2$ & 2.708344826299720090001... & 10m 10s \\
 \hline
\end{tabular}
\end{center}
\end{example}

There are, however, many examples within our class that do not terminate after days, since the creative telescoping computation becomes too large. Therefore, understanding how the order and degree of the Picard--Fuchs operators change as the input polynomials vary is of interest. 
We make the following observation:
\begin{example}[Picard--Fuchs operator of an $\ell_p$-ball]
    We computed the Picard--Fuchs operator of a single $\ell_p$-ball which annihilates the volumes of the $x_i$-slices. We performed this computation for the following values of $n$ and $p$:
    \begin{align*}
        n = 2,3 & \qquad p \leq 36\\
        n = 4 & \qquad p \leq 28.
    \end{align*}
    With the exception of $n=4$, $p=2$, 
    the Picard--Fuchs operator $P_{x_i}$ 
    has the form
    \begin{align}
        P_{x_1} = (1 - x_i^p) \partial_{x_i} + (n-1)x_i^{p-1}.
    \end{align}
\end{example} 
We expect that a general statement can be made here. The case of $n=4,p=2$ is discussed in Example \ref{eg:spurious poles}.

In the next two examples, we consider the intersection of two convex bodies in~$\mathbb R^2$. Keeping the first one fixed and moving the second, we take note of changes in the order and the degree of the computed Picard--Fuchs operators.

\begin{example}[Translating two $\ell_4$-balls]
    We consider two $\ell_4$-balls in $\mathbb R^2$ centered at $\mu_1 = (0,0)$ and $\mu_2 = (a,0)$ for varying $a$ and compute the Picard--Fuchs operator in the deformation parameter~$t$. 

    For generic values (such as $27/45, 21/20, -1/7, 1$) of $a$, the operator $P_t$ has order $6$ and degree $18$. 
    However, for special values of $a$, the degree drops. For $a = 2$ and $a = -2$, where the intersection vanishes, we obtain $\ord(P_t) = 6$ and $\deg(P_t) = 17$. We ignore the case of $a = 0$ since that corresponds to a single $\ell_p$-ball.

\end{example}

\begin{example}[Rotating two Euclidean balls]
    Consider in $\mathbb R^2$ two $\ell_2$-balls with centers $\mu_1 = (0,0)$ and $\mu_2 = (a,b)$, where $\mu_2$ is a rational point on the unit circle. The Picard--Fuchs operator $P_t$ stays invariant under the choice of $\mu_2$.
    This is because the volume is invariant under rotation. 
    However, we see from the following table that the Picard--Fuchs operator $P_{x_1}$ annihilating the $x_1$-slices of the deformation $C_t$ for $t = 1/50$ has the same degree and order if $a, b \neq 0$. If $a$ or $b$ is equal to $0$ then both the degree and order drop. 

    \begin{center}
    \begin{tabular}{ |c|c|c| } 
     \hline
     $a,b$ & $\ord(P_{x_1})$ & $\deg(P_{x_1})$ \\ 
     \hline
     $1,0$ & 4 & 10 \\ 
     $0,1$ & 3 & 9 \\ 
     $-3/5,4/5$ & 5 & 17 \\ 
     $15/17,8/17$ & 5 & 17 \\ 
     $8/10,-6/10$ & 5 & 17\\ 
     \hline
    \end{tabular}
    \end{center}
\end{example}

Since the volume of a convex body in $\mathbb{R}^n$ is invariant under the action of $\operatorname{SL}_n(\mathbb Q)$ and translations, reducing the number of terms of the input polynomials by such a transformation appears to be a strategy to obtain a Picard--Fuchs operator of a lower order.

\vspace{0.3cm}

We conclude with a comment on the choice of the creative telescoping algorithm. In our current implementation,  creative telescoping is carried out via the implementation of Chyzak's algorithm~\cite{chyzak'salgorithm} in the \texttt{ore\_algebra} package~\cite{KauersMezzarobba2019}. While in Theorem~\ref{thm:takayama}, the integration ideal is computed over the Weyl algebra $D_{t,\mathbf{x}}$, Chyzak's algorithm actually computes elements in
\begin{equation}
     (J + \partial_{x_1}R_{t, \textbf{x}} + \cdots +\partial_{x_n} R_{t,\textbf{x}} ) \quad \cap  \quad  R_t,
\end{equation}
so that differential operators with rational coefficients are also allowed.
This leads in some cases to unwanted behaviors. 
If the certificates of the creative telescoping process have {\em spurious} poles, i.e. poles other than $f$ and powers of $f$, 
then the proof of Theorem~\ref{thm:takayama} does not hold. 
As a consequence, the computed Picard--Fuchs operator may not actually annihilate the period integral as the following extreme example, found in \cite{Mezzarobba-volumes-repo}, shows:

\begin{example}[Spurious poles]\label{eg:spurious poles}
    Consider a single 
    Euclidean ball in $\mathbb{R}^4$ defined by $C =  \{x \in \mathbb{R}^4 \mid f(x)>0\}$ for $f = 1 - (x_1^2 + x_2^2 + x_3^2 + x_4^2)$. Then, $$A = \frac{(\partial_1\bullet f)x_1}{f} = -\frac{2 x_1^2}{f}$$ denotes the rational function whose periods describe the volume of, for example, the $x_4$-slices of $C$. However, the $R_{x_4}$ ideal
    \begin{equation}
        I = (\operatorname{Ann}_{R_{\mathbf{x}}}(A) + \partial_{x_1}R_{ \textbf{x}} + \partial_{x_2}R_{ \textbf{x}} + \partial_{x_3}R_{ \textbf{x}} ) \quad \cap  \quad  R_{x_4}
    \end{equation}
    contains the operator $1$, since the rational function $A$ can be decomposed as a sum of derivatives 
    \begin{align*}
        3A = \partial_{1}\bullet\frac{ 2x_1^3}{f} + \partial_{2}\bullet \frac{ 2x_1^4x_2}{f\cdot(x_4^2+x_1^2-1)}+\partial_{3} \bullet \frac{2x_3x_1^4}{f\cdot(x_4^2+x_1^2-1)}.
    \end{align*}
    From the above equation, we see that for the telescoper $P = 1$ and certificates
    $$Q_1=\frac{x_1}{3},\quad Q_2=\frac{x_1^2 x_2}{3(x_4^2 + x_2^2 - 1)},\quad Q_3 = \frac{x_1^2 x_3}{3(x_4^2 + x_2^2 - 1)},$$ the differential operator $$P - \partial_1 Q_1 - \partial_2 Q_2 - \partial_3 Q_3$$ annihilates $A$. However, $P$ clearly does not annihilate the periods of $A$, since the slices of $C$ have non-zero volume. This does not contradict Theorem~\ref{thm:takayama} since here the certificates do not lie in $D_\mathbf{x}$. 
\end{example}

A more subtle example of unexpected behavior is the following:
\begin{example}[Singular locus of Picard--Fuchs operator]\label{eg:problem2}
    Consider two $\ell_2$-balls centered at $\mu_1 = (0,0,0,0)$ and $\mu_2 =(1,0,0,0)$ as in Example~\ref{eg:hypersurfregions}, the projection of the deformation for $t = 1/1000$ onto the $x_0$ axis gives the Picard--Fuchs operator $1$. 

    If we change the order of projection and project the deformation first onto the $x_1$ axis, we get a Picard--Fuchs operator of degree $10$ and order $3$. For $x_1 = -7/10$, we project onto the $x_0$ axis. Here we encounter a different problem. The singular locus of the operator should contain the critical values of the projection due to the fact that at the critical values the number of connected components of the fiber changes. This shows us that the computed operator does not annihilate the period. 
\end{example}

To avoid issues such as those in the above examples, the implementation allows the input of a custom order of projections via the routine \texttt{ProjectionVariable}. For two $\ell_2$ balls in $\mathbb R^4$, the projection order in Example \ref{eq:l24} avoids these issues. To avoid these issues altogether, one could instead call a creative telescoping algorithm that operates in $D_\mathbf{x}$, such as in MultivariateCreativeTelescoping.jl \cite{brochet2025fastermultivariateintegrationdmodules} (MCT.jl). 
As a first step in this direction, we have made use of an experimental interface with MCT.jl in order to compute in Example~\ref{ex:mathematicaL4L4R4} the  volume of the intersection of two $\ell_4$-balls in $\mathbb{R}^4$ with centers $(0,0,0,0)$ and $(1,0,0,0)$. We determined the volume to $95$ digits in $15$ hours.

We conclude by remarking that even though computation of Picard--Fuchs operators quickly becomes infeasible with increasing dimension and degree, there are examples, see Example~\ref{eg:mathematica2}, where symbolic integration is inconclusive, but where our implementation outperforms numerical integration over CAD cells.

\bibliographystyle{ACM-Reference-Format}
\bibliography{ISSAC_references}

\section*{Additional Experiments}

\begin{example}\label{ex:mathematicaL4L4R4}
    We determine the volume of the intersection of two $\ell_4$-balls in $\mathbb{R}^4$ defined by  $1 - (x^4 + y^4 + z^4 + w^4) > 0$ and $1 - ((x - 1)^4 + y^4 + z^4 + w^4) > 0$. This was computed on a personal computer with an Apple M3 chip and $32$ GB of memory and using an experimental interface to MultivariateCreativeTelescoping.jl in $15$ hours up to $95$ digits as
    \begin{align*} 4.&37856654871924288558479945106166817027285180219207\\&050298477722007062424113388554779398481745775\ldots\,
    \end{align*}
\end{example}

\begin{example}\label{eg:mathematica}
    We computed the volume of the convex planar region defined by $1 - (x^2 + y^2)^4 - 1/10(x^4y^2 + x^2y^4)^2>0$ in 62s with our method on an Apple M1 chip with 8GB of memory to 300 digits to be
    \begin{align*}
    3.&13976124057952680602946138619253044761966453475542\\&729206675302931045918949716548969288957791094\ldots\,
    \end{align*} 
    This is an example where Mathematica \cite{Mathematica14} fails to return anything, as the kernel restarts. The generic CAD is computed as follows 
    \begin{lstlisting}
{t, gcd} = Timing[GenericCylindricalDecomposition[
    1 - (x^2 + y^2)^4 - 1/10*(x^4 y^2 + x^2 y^4)^2 
    > 0, {x, y}]]
    \end{lstlisting}
    and returns timing in seconds, and a description of the region.
    \begin{lstlisting}
{0.026082, {-1 < x < 1 && Root[-10 + 10 x^8 + 
    40 x^6 #1^2 + (60 x^4 + x^8) #1^4 + (40 x^2 
    + 2 x^6) #1^6 + (10 + x^4) #1^8 &, 1] < y < 
    Root[-10 + 10 x^8 + 40 x^6 #1^2 + (60 x^4 + 
    x^8) #1^4 + (40 x^2 + 2 x^6) #1^6 
    + (10 + x^4) #1^8 &, 2], False}}
    \end{lstlisting}
    Neither symbolic nor numerical integration of this region is completed by Mathematica, though the reason is unclear. 
\end{example}
\begin{example}\label{eg:mathematica2}
    We compute the volume of the convex planar region defined by
    $1 - (x^2 + y^2)^3 - 3/10(x^4y^2 + x^2y^4)>0$.
    In Mathematica, symbolic integration fails, but the volume can be computed by numerical integration over the  algebraic cell returned by the generic CAD as follows:
    \begin{lstlisting}
N[Integrate[
    Root[-10 + 10 x^6 + 33 x^4 #1^2 + 33 x^2 #1^4 
    + 10 #1^6 &, 2] - 
    Root[-10 + 10 x^6 + 33 x^4 #1^2 + 33 x^2 #1^4 
    + 10 #1^6 &, 1], {x, -1, 1}], digits]
    \end{lstlisting}
    
    The following table shows the timings of obtaining the volume of the region up to a given precision using the numerical integration approach in comparison to our implementation, on an Apple M1 chip with 8GB of memory. 
    
    \begin{center}
    \begin{tabular}{ |c|c|c| } 
     \hline
     Digits & Mathematica & 
     Our Implementation \\ 
     \hline
     50 & 10.4s & 0.8s \\ 
     100 & 12.3s & 0.9s \\ 
     200 & 16.9s & 1s \\ 
     \hline
    \end{tabular}
    \end{center}
    
\end{example}
\end{document}